\documentclass[letterpaper, 10 pt, conference]{ieeeconf} 
\IEEEoverridecommandlockouts
\usepackage{graphics}
\usepackage{epsfig}
\usepackage{mathptmx}
\usepackage{times}
\usepackage{amsmath}
\usepackage{amssymb}
\usepackage{mathtools}
\let\labelindent\relax
\usepackage{enumitem}
\usepackage{cite}
\usepackage{tikz}
\usetikzlibrary{shapes, arrows, arrows.meta, decorations.markings}

\newtheorem{theorem}{Theorem}[section]
\newtheorem{lemma}[theorem]{Lemma}
\newtheorem{corollary}[theorem]{Corollary}
\newtheorem{proposition}[theorem]{Proposition}

\newtheorem{definition}[theorem]{Definition}

\DeclareMathOperator{\diverg}{div}
\DeclareMathOperator{\Lip}{Lip}
\DeclareMathOperator{\MFG}{MFG}
\DeclareMathOperator{\OCP}{OCP}
\DeclareMathOperator{\Adm}{Adm}
\DeclareMathOperator{\Opt}{Opt}

\DeclarePairedDelimiter{\abs}{\lvert}{\rvert}

\title{\LARGE \bf
On the characterization of equilibria of nonsmooth minimal-time mean field games with state constraints
}

\author{Saeed Sadeghi Arjmand$^1$ and Guilherme Mazanti$^2$%
\thanks{$^1$CMLS, \'Ecole Po\-ly\-tech\-ni\-que, CNRS, Universit\'e Paris-Saclay, 91128, Palaiseau, France, \& Uni\-ver\-si\-t\'e Paris-Saclay, CNRS, CentraleSup\'elec, Inria, Laboratoire des signaux et syst\`emes, 91190, Gif-sur-Yvette, France.
        {\tt\small saeed.\allowbreak{}sadeghi\allowbreak{}-arjmand@polytechnique.edu}}%
\thanks{$^2$Universit\'e Paris-Saclay, CNRS, CentraleSup\'elec, Inria, Laboratoire des signaux et syst\`emes, 91190, Gif-sur-Yvette, France.
        {\tt\small guilherme.\allowbreak{}mazanti@inria.fr}}%
}

\begin{document}

\newlist{hypotheses}{enumerate}{1}
\setlist[enumerate, 1]{label={\textup{(\alph*)}}, leftmargin=*}
\setlist[hypotheses]{label={\textup{(H\arabic*)}}, leftmargin=*}

\maketitle
\thispagestyle{empty}
\pagestyle{empty}

\begin{abstract}
In this paper, we consider a first-order deterministic mean field game model inspired by crowd motion in which agents moving in a given domain aim to reach a given target set in minimal time. To model interaction between agents, we assume that the maximal speed of an agent is bounded as a function of their position and the distribution of other agents. Moreover, we assume that the state of each agent is subject to the constraint of remaining inside the domain of movement at all times, a natural constraint to model walls, columns, fences, hedges, or other kinds of physical barriers at the boundary of the domain. After recalling results on the existence of Lagrangian equilibria for these mean field games and the main difficulties in their analysis due to the presence of state constraints, we show how recent techniques allow us to characterize optimal controls and deduce that equilibria of the game satisfy a system of partial differential equations, known as the mean field game system.
\end{abstract}

\section{Introduction}

The concept of mean field games (referred to as ``MFGs'' in this paper for short) was first introduced around 2006 by two independent groups, P.~E.~Caines, M.~Huang, and R.~P.~Malhamé \cite{Huang2003Individual, HuangMalhame}, and J.-M.~Larsy and P.-L.~Lions\cite{LasryLionsfr1, LasryLionsfr2}, motivated by problems in economics and engineering and building upon previous works on games with infinitely many agents such as \cite{Aumann1974Values, Jovanovic1988Anonymous}. Roughly speaking, MFGs are game models with a continuum of indistinguishable, rational agents influenced only by the average behavior of other agents, and the typical goal of their analysis is to characterize their equilibria. We refer the interested reader to \cite{Carmona2018Probabilistic} for more details on MFGs.

In this paper, we study an MFG model inspired by crowd motion in which agents want to reach a given target set in minimal time, their maximal speed being bounded in terms of the distribution of other agents and their state being constrained to remain in a given bounded set. Modeling and analysis of crowd motion have been the subject of a large number of works from different perspectives, such as \cite{Cristiani2014Multiscale, Gibelli2018Crowd, Muntean2014Collective}, and some deterministic and stochastic MFG models have been already proposed, for instance, in \cite{Bagagiolo2019Optimal, Burger2013Mean, Dweik2020Sharp, Lachapelle2011Mean, Mazanti2019Minimal, SadeghiMulti}. MFG models for crowd motion usually try to capture strategic choices of the crowd based on the rational anticipation by an agent of the behavior of others.

The MFG model we consider in this paper is that of \cite{Mazanti2019Minimal}, its detailed description is provided in Section~\ref{SecModel}. An important feature of the model from \cite{Mazanti2019Minimal} which renders its analysis more delicate is the fact that the final time of the movement of an agent is not prescribed, but it is part of the agent's optimization criterion. Reference \cite{Mazanti2019Minimal} establishes existence of equilibria of the considered MFG model, but additional properties, such as characterization of optimal controls and characterization of equilibria through the system of PDEs known as MFG system, are only obtained in \cite{Mazanti2019Minimal} under the restrictive assumption that the target set of the agents is the whole boundary of the domain, which avoids the presence of state constraints in the minimal-time optimal control problem solved by each agent. The main contribution of the present paper is to characterize optimal controls and obtain the MFG system without such a restrictive assumption.

The major difficulty in analyzing optimal control problems with state constraints is that their value functions may fail to be semiconcave (see, e.g., \cite[Example~4.4]{CannarsaPiermarcoCastelpietra}), the latter property being important in the characterization of optimal controls (see, e.g., \cite{Cannarsa2004Semiconcave}). In this paper, we rely instead on the techniques introduced in \cite{SadeghiMulti} to characterize optimal controls, which do not rely on the semiconcavity of the value function and also allow for weaker regularity assumptions on the dynamics of agents. In order to obtain the classical necessary optimality conditions from Pontryagin Maximum Principle (PMP) under state constraints and few regularity assumptions, we rely on the nonsmooth PMP from \cite{Clarke} and make use of the technique from \cite{CannarsaCastelpietraCardaliaguet} to deal with state constraints. We also refer the interested reader to \cite{CannarsaPiermarco, Cannarsa2019C11, CannarsaMean} for an alternative approach for dealing with other MFG models with state constraints.

Our main results are Theorem~\ref{thm Ut_0,x_0} and its Corollary~\ref{coro normalized}, which provide the characterization of optimal controls, and Theorem~\ref{Thm MFG system}, which relies on that characterization to show that equilibria of the MFG satisfy a suitable system of PDEs.

This paper is organized as follows. Section~\ref{SecModel} presents the MFG model and the definition of equilibria and recalls the major previous results useful for this paper. We then study, in Section~\ref{SecOptControl}, the corresponding optimal control problem, providing the characterization of optimal controls under state constraints. This analysis is finally used in Section~\ref{sec MFG system} to show that equilibria of the mean field game satisfy a system of PDEs made of a continuity equation on the density of agents and a Hamilton--Jacobi equation on the value function of the corresponding optimal control problem.

\smallskip

\noindent\emph{Notation.} In this paper, $d$ denotes a positive integer, the set of nonnegative real numbers is denoted by $\mathbb R_+$, $\mathbb R^d$ is endowed with its usual Euclidean norm $\abs{\cdot}$, and $\mathbb S^{d-1}$ denotes the unit sphere in $\mathbb R^d$. For $R \geq 0$, we use $B_R$ to denote the closed ball in $\mathbb R^d$ centered at the origin and with radius $R$. The closure of a set $A \subset \mathbb R^d$ is denoted by $\bar A$.

Given a Polish space $X$, the set of all Borel probability measures on $X$ is denoted by $\mathcal P(X)$, which is always assumed to be endowed with the weak convergence of measures. When $X$ is endowed with a complete metric $d$ with respect to which $X$ is bounded, we assume that $\mathcal P(X)$ is endowed with the Wasserstein distance $W_1$ defined by $W_1(\mu, \nu) = \sup_\phi \int_{X} \phi(x) d(\mu - \nu)(x)$, where the supremum is taken over all $1$-Lipschitz continuous functions $\phi: X \to \mathbb R$.

Given two metric spaces $X, Y$ and $M \geq 0$, we denote by $C(X, Y)$ the set of all continuous functions $f: X \to Y$, $\Lip(X, Y)$ the subset of $C(X, Y)$ of all Lipschitz continuous functions, and by $\Lip_M(X, Y)$ the subset of $\Lip(X, Y)$ of those functions whose Lipschitz constant is bounded by $M$. When $X = \mathbb R_+$, the above sets are denoted simply by $C(Y)$, $\Lip(Y)$, and $\Lip_M(Y)$, respectively.

For compact $A \subset \mathbb R^d$, the space $C(A)$ is assumed to be endowed with the topology of uniform convergence on compact sets, with respect to which $C(A)$ is a Polish space. For $t\in \mathbb R_+$, we let $e_t: C(A) \to A$ denote the evaluation map, defined for $\gamma \in C(A)$ by $e_t(\gamma) = \gamma(t)$.

If $X$ and $Y$ are two metric spaces endowed with their Borel $\sigma$-algebras, $f: X \to Y$ is Borel measurable, and $\mu$ is a Borel measure in $X$, we denote the pushforward of $\mu$ through $f$ by $f_{\#} \mu$, i.e., $f_{\#} \mu$ is the Borel measure in $Y$ defined by $f_{\#} \mu(A) = \mu(f^{-1}(A))$ for every Borel subset $A$ of $Y$.

\section{Description of the MFG model and previous results}
\label{SecModel}

In this paper, we fix an open and bounded set $\Omega \subset \mathbb R^d$ and we let $\Gamma \subset \bar\Omega$ be a closed nonempty set. We shall always assume that $\Omega$ satisfies the following hypothesis.

\begin{hypotheses}
\item\label{HypoOmega-GeoDist} There exists $D > 0$ such that, for every $x, y \in \bar\Omega$, there exists a curve $\gamma$ included in $\bar\Omega$ connecting $x$ to $y$ and of length at most $D \abs{x - y}$.
\end{hypotheses}

Note that \ref{HypoOmega-GeoDist} means that the geodesic distance in $\bar\Omega$ is equivalent to the usual Euclidean distance.

Let $K: \mathcal P(\bar\Omega) \times \bar\Omega \to \mathbb R_+$ be bounded and $m_0 \in \mathcal P(\bar\Omega)$. The MFG considered in this paper, denoted by $\MFG(K, m_0)$, is described as follows. A population of agents moves in $\bar\Omega$ and is described at time $t \geq 0$ by a time-dependent probability measure $m_t\in \mathcal{P}(\bar{\Omega})$, where $m_0$ is the prescribed probability measure. Each agent wants to choose their trajectory in $\bar{\Omega}$ in order to reach the target set $\Gamma$ in minimal time, with the constraints that the agent must remain in $\bar\Omega$ at all times and that their maximal velocity at time $t$ and position $x$ is given by $K(m_t, x)$, i.e., the trajectory $\gamma$ of each agent solves the control system
\[
\dot\gamma(t) = K(m_t, \gamma(t)) u(t), \qquad \gamma(t) \in \bar\Omega,\, u(t) \in B_1,
\]
where $u(t)$ is the control of the agent at time $t$, chosen in order to minimize the time to reach $\Gamma$. We also assume that, once an agent reaches $\Gamma$, they stop.

Note that agents interact through the maximal velocity $K(m_t, x)$, which depends on the distribution of agents $m_t$ at time $t$. Hence, the trajectory of an agent depends on $m_t$ and, on the other hand, $m_t$ is determined by the trajectories of the agents. We are interested here in \emph{equilibrium} situations, i.e., situations in which, starting from an evolution $t \mapsto m_t$, the optimal trajectories chosen by the agents induce an evolution of the initial distribution $m_0$ that coincides with $t \mapsto m_t$ (a more mathematically precise definition of equilibrium is provided in Definition~\ref{DefiEquilibrium} below).

The interaction term $K(m_t, x)$ can be used to model congestion phenomena in crowd motion by choosing a function $K$ such that $K(m_t, x)$ is small when $m_t$ is ``large'' around $x$, which means that it is harder to move on more crowded regions. For instance, $K$ may be chosen as
\[
K(\mu, x) = g\left(\int_{\bar\Omega} \chi(x - y) d\mu(y)\right),
\]
where $\chi$ is a convolution kernel representing the region around an agent at which they look in order to evaluate local congestion and $g$ is a decreasing function. Note that we do \emph{not} assume this specific form for $K$ in this paper.

\subsection{The auxiliary optimal control problem}

Let us describe the optimal control problem solved by each agent of the game. Let $k: \mathbb R_+ \times \bar\Omega \to \mathbb R_+$ be bounded and consider the control system
\begin{equation}
\label{Control sys}
\dot\gamma(t) = k(t, \gamma(t)) u(t), \qquad \gamma(t) \in \bar\Omega,\, u(t) \in B_1,
\end{equation}
where $\gamma(t)$ is the state and $u(t)$ is the control at time $t \geq 0$. An absolutely continuous function $\gamma: \mathbb R_+ \to \bar\Omega$ is said to be an \emph{admissible trajectory} of \eqref{Control sys} if it satisfies \eqref{Control sys} for a.e.\ $t \geq 0$ for some measurable $u: \mathbb R_+ \to B_1$, and the corresponding function $u$ is said to be the \emph{control} associated with $\gamma$. The set of all admissible trajectories for \eqref{Control sys} is denoted by $\Adm(k)$. For $\gamma \in \Adm(k)$ and $t_0 \geq 0$, the \emph{first exit time after $t_0$ of $\gamma$} is the value $\tau_\Gamma(t_0, \gamma) = \inf\{T \geq 0 \mid \gamma(t_0 + T) \in \Gamma\}$, with the convention that $\inf\varnothing = +\infty$.

We consider the optimal control problem $\OCP(k)$ defined as follows: given $(t_0, x_0) \in \mathbb R_+ \times \bar\Omega$, solve
\begin{equation}
\label{min exit time}
\inf_{\substack{\gamma \in \Adm(k) \\ \gamma(t_0) = x_0}} \tau_\Gamma(t_0, \gamma).
\end{equation}
A trajectory $\gamma$ attaining the above infimum is called an \emph{optimal trajectory for $(k, t_0, x_0)$} and its associated control $u$ is called an \emph{optimal control}. Note that an optimal control $u$ for $\gamma$ remains optimal if it is modified outside of the interval $[t_0, t_0 + \tau_\Gamma(t_0, \gamma)]$. In order to avoid any ambiguity, we always assume, in this paper, that optimal controls are equal to $0$ in the intervals $[0, t_0)$ and $(t_0 + \tau_\Gamma(t_0, \gamma), +\infty)$, and in particular optimal trajectories are constant in the intervals $[0, t_0]$ and $[t_0 + \tau_\Gamma(t_0, \gamma), +\infty)$. The set of all optimal trajectories for $(k, t_0, x_0)$ is denoted by $\Opt(k, t_0, x_0)$.

The link between $\MFG(K, m_0)$ and $\OCP(k)$ is that, given an evolution of agents $t \mapsto m_t$, each agent of the crowd solves $\OCP(k)$ with $k(t, x) = K(m_t, x)$. The optimal control problem $\OCP(k)$ is a minimum time problem, which is a classical problem in control theory for which several results are available (see, e.g., \cite[Chapter~8]{Cannarsa2004Semiconcave} and \cite[Chapter~IV]{Bardi1997Optimal}). A classical tool in the analysis of optimal control problems is the value function $\varphi: \mathbb R_+ \times \bar\Omega \to \mathbb R_+$, defined for $(t_0, x_0) \in \mathbb R_+ \times \bar\Omega$ by setting $\varphi(t_0, x_0)$ to be equal to the value of the infimum in \eqref{min exit time}.

We shall consider in this paper $\OCP(k)$ under the following assumption.

\begin{hypotheses}[resume]
\item\label{HypoOCP-k} We have $k \in \Lip(\mathbb R_+ \times \bar\Omega, \mathbb R_+)$ and there exist positive constants $K_{\min}, K_{\max}$ such that $k(t, x) \in [K_{\min}, K_{\max}]$ for every $(t, x) \in \mathbb R_+ \times \bar\Omega$.
\end{hypotheses}

We collect in the next proposition classical results on $\OCP(k)$ that will be of use in this paper (see, e.g., \cite[Section~4]{Mazanti2019Minimal}).

\begin{proposition}
\label{PropOCP}
Consider $\OCP(k)$ under hypotheses \ref{HypoOmega-GeoDist} and \ref{HypoOCP-k} and let $(t_0, x_0) \in \mathbb R_+ \times \bar\Omega$.
\begin{enumerate}
\item The set $\Opt(k, t_0, x_0)$ is nonempty.
\item The value function $\varphi$ is Lipschitz continuous on $\mathbb R_+ \times \bar\Omega$.
\item\label{PropOCP-DPP} For every $\gamma \in \Adm(k)$ such that $\gamma(t_0) = x_0$, we have, for every $h \geq 0$,
\begin{equation}
\label{eq:DPP}
\varphi(t_0 + h, \gamma(t_0 + h)) + h \geq \varphi(t_0, x_0),
\end{equation}
with equality if $\gamma \in \Opt(k, t_0, x_0)$ and $h \in [0, \tau_\Gamma(t_0, \gamma)]$. Conversely, if $\gamma \in \Adm(k)$ satisfies $\gamma(t_0) = x_0$, if $\gamma$ is constant on $[0, t_0]$ and on $[t_0 + \tau_\Gamma(t_0, \gamma), +\infty)$, and if equality holds in \eqref{eq:DPP} for every $h \in [0, \tau_\Gamma(t_0, \gamma)]$, then $\gamma \in \Opt(k, t_0, x_0)$.
\item\label{PropOCP-HJ} The value function $\varphi$ satisfies the Hamilton--Jacobi equation
\begin{equation}
\label{H-J equation}
-\partial_t \varphi(t,x)+\abs{\nabla \varphi(t,x)} k(t, x) - 1 = 0
\end{equation}
in the following sense: $\varphi$ is a viscosity subsolution of \eqref{H-J equation} in $\mathbb R_+ \times (\Omega \setminus \Gamma)$, a viscosity supersolution of \eqref{H-J equation} in $\mathbb R_+ \times (\bar\Omega \setminus \Gamma)$, and satisfies $\varphi(t, x) = 0$ for every $(t, x) \in \mathbb R_+ \times \Gamma$.
\item\label{PropOCP-OptimalControl} If $\gamma \in \Opt(k, t_0, x_0)$, $t \in [t_0, t_0 + \varphi(t_0, x_0))$, $\gamma(t) \in \Omega \setminus \Gamma$, and $\varphi$ is differentiable at $(t, \gamma(t))$, then $\abs{\nabla\varphi(t, \gamma(t))} \neq 0$ and
\[
\dot\gamma(t) = - k(t, \gamma(t)) \frac{\nabla\varphi(t, \gamma(t))}{\abs{\nabla\varphi(t, \gamma(t))}}.
\]
\end{enumerate}
\end{proposition}

\subsection{Lagrangian equilibria and their existence}

In this paper, we study $\MFG(K, m_0)$ in a Lagrangian setting, in which the evolution of agents is described by a measure $Q \in \mathcal P(C(\bar\Omega))$ in the space of all continuous trajectories $C(\bar\Omega)$. This classical approach in optimal transport has become widely used in the analysis of MFGs with deterministic trajectories in recent years (see, e.g., \cite{CannarsaPiermarco, CardaliaguetPierre, CardaliaguetPierre2, Mazanti2019Minimal, Dweik2020Sharp
}). Note that the distribution $m_t$ of agents at time $t \geq 0$ can be retrieved from $Q$ using the evaluation map $e_t$ by $m_t = {e_t}_{\#} Q$. The definition of an equilibrium of $\MFG(K, m_0)$ is formulated in the Lagrangian setting as follows.

\begin{definition}\label{DefiEquilibrium}
Consider $\MFG(K, m_0)$. A measure $Q \in \mathcal{P}(C(\bar{\Omega}))$ is called a \emph{Lagrangian equilibrium} (or simply \emph{equilibrium}) of $\MFG(K, m_0)$ if ${e_0}_{\#} Q = m_0$ and $Q$-almost every $\gamma \in C(\bar{\Omega})$ is an optimal curve for $(k, 0, \gamma(0))$, where $k:\mathbb R_+ \times \bar{\Omega} \to \mathbb R_+$ is defined by $k(t,x) = K(e_{t\#}Q, x)$.
\end{definition}

The next assumption is the counterpart of \ref{HypoOCP-k} for $\MFG(K, m_0)$.

\begin{hypotheses}[resume]
\item\label{HypoMFG-K} We have $K \in \Lip(\mathcal P(\bar\Omega) \times \bar\Omega, \mathbb R_+)$ and there exist positive constants $K_{\min}, K_{\max}$ such that $K(\mu, x) \in [K_{\min}, K_{\max}]$ for every $(\mu, x) \in \mathcal P(\bar\Omega) \times \bar\Omega$.
\end{hypotheses}

We recall in the next theorem the main result of \cite{Mazanti2019Minimal} concerning existence of equilibria.
\begin{theorem}
Consider $\MFG(K, m_0)$ under assumptions \ref{HypoOmega-GeoDist} and \ref{HypoMFG-K}. Then there exists an equilibrium $Q \in \mathcal{P}(C(\bar{\Omega}))$ for $\MFG(K, m_0)$.
\end{theorem}

\section{Further properties of the optimal control problem}
\label{SecOptControl}

We provide in this section further properties of $\OCP(k)$ with the aim of providing a characterization of optimal controls. For that purpose, we assume the following additional hypothesis on $\Omega$.

\begin{hypotheses}[resume]
\item\label{HypoOmega-C11} The boundary $\partial\Omega$ is a compact $C^{1,1}$ manifold.
\end{hypotheses}

We will denote in the sequel by $d^{\pm}$ the signed distance to $\partial\Omega$, defined by $d^{\pm}(x) = d(x, \Omega) - d(x, \mathbb R^d \setminus\Omega)$, where $d(x, A) = \inf_{y \in A} \abs{x - y}$ for $A \subset \mathbb R^d$. Recall that, under assumption \ref{HypoOmega-C11}, $d^{\pm}$ is $C^{1, 1}$ in a neighborhood of $\partial\Omega$, its gradient has unit norm, and $\nabla d^{\pm}$ is a Lipschitz continuous function extending the exterior normal vector field of $\Omega$ to a neighborhood of $\partial\Omega$ (see, e.g., \cite{Delfour1994Shape}).

\subsection{Consequences of Pontryagin Maximum Principle}

In order to obtain additional properties of optimal trajectories, we apply Pontryagin Maximum Principle to a modified optimal control problem without state constraints, following the techniques from \cite{CannarsaCastelpietraCardaliaguet}. Assume that $k$ satisfies \ref{HypoOCP-k} and is extended to a Lipschitz continuous function defined on $\mathbb R_+ \times \mathbb R^d$. We also assume, with no loss of generality, that the extension of $k$ is $C^1$ on $\mathbb R_+ \times (\mathbb R^d \setminus \bar\Omega)$. For $\epsilon > 0$, define $k_\epsilon: \mathbb R_+ \times \mathbb R^d \to \mathbb R_+$ by
\begin{equation}
\label{eq:k-eps}
k_\epsilon(t, x) = k(t, x) \left(1 - \frac{1}{\epsilon} d(x, \Omega)\right)_+,
\end{equation}
where $a_+$ is defined by $a_+ = \max(0, a)$ for $a \in \mathbb R$. Consider the control system
\begin{equation}
\label{Control sys epsilon}
\dot\gamma_\epsilon(t) = k_\epsilon(t, \gamma_\epsilon(t)) u_\epsilon(t), \qquad \gamma_\epsilon(t) \in \mathbb R^d,\, u_\epsilon(t) \in B_1
\end{equation}
and the optimal control problem $\OCP_\epsilon(k_\epsilon)$ of, given $(t_0, x_0) \in \mathbb R_+ \times \mathbb R^d$, finding a measurable control $u_\epsilon$ such that the corresponding trajectory $\gamma_\epsilon$ solving \eqref{Control sys epsilon} reaches $\Gamma$ in minimal time. The next lemma states the main consequences of Pontryagin Maximum Principle when applied to $\OCP_\epsilon(k_\epsilon)$.

\begin{lemma}\label{Conse Pontryagin}
Consider $\OCP_\epsilon(k_\epsilon)$ under assumptions \ref{HypoOmega-GeoDist}, \ref{HypoOCP-k}, and \ref{HypoOmega-C11} and with $k_\epsilon$ defined by \eqref{eq:k-eps}. Let $(t_0,x_0) \in \mathbb R_+ \times \bar{\Omega}$, $\gamma_{\epsilon}$ be an optimal trajectory for $\OCP_\epsilon(k_\epsilon)$, $T_\epsilon$ be the first exit time of $\gamma_\epsilon$, and $u_\epsilon: [t_0, t_0+T_{\epsilon}] \to B_1$ be an optimal control associated with $\gamma_{\epsilon}$. Then $d(\gamma_{\epsilon}(t), \Omega) < \epsilon$ for every $t\in [t_0,t_0+T_{\epsilon}]$ and there exist $\lambda_{\epsilon}\in \{0,1\}$ and absolutely continuous functions $p_{\epsilon}:[t_0,t_0+T_{\epsilon}]\to \mathbb R^d$ and $q_{\epsilon}:[t_0,t_0+T_{\epsilon}]\to \mathbb R$ such that, for a.e.\ $t \in [t_0, t_0 + T_\epsilon]$,
\begin{subequations}
\label{eq:cons-Pontryagin}
\begin{align}
\dot q_\epsilon(t) & \in \abs{p_\epsilon(t)} \pi_1 \partial^C k_\epsilon(t, \gamma_\epsilon(t)), \\
-\dot p_\epsilon(t) & \in \abs{p_\epsilon(t)} \pi_2 \partial^C k_\epsilon(t, \gamma_\epsilon(t)), \label{eq:cons-Pontryagin:p-epsilon} \\
q_\epsilon(t) & = \abs{p_\epsilon(t)} k_\epsilon(t, \gamma_\epsilon(t)) - \lambda_\epsilon, \label{eq:cons-Pontryagin:hamiltonian}\\
\max_{w \in B_1} p_\epsilon(t) \cdot w & = p_\epsilon(t) \cdot u_\epsilon(t), \label{eq:cons-Pontryagin:maximization} \\
q_\epsilon(t_0 + T_\epsilon) & = 0, \label{eq:cons-Pontryagin:finalTime} \\
\lambda_\epsilon + \max_{t \in [t_0, t_0 + T_\epsilon]} \abs{p_\epsilon(t)} & > 0, \label{eq:cons-Pontryagin:nontrivial}
\end{align}
\end{subequations}
where $\partial^C$ denotes Clarke's gradient (see \cite{Clarke} for its definition) and $\pi_1, \pi_2$ are the projections onto the first and second factors of the product $\mathbb R \times \mathbb R^d$, respectively.
\end{lemma}

The proof of Lemma~\ref{Conse Pontryagin} is standard and can be carried out by showing first that $d(\gamma_{\epsilon}(t), \Omega) < \epsilon$ for every $t\in [t_0,t_0+T_{\epsilon}]$, which holds since, otherwise, $\gamma_\epsilon$ would belong, for some time, to a region outside of $\bar\Omega$ where $k_\epsilon$ is identically zero, and hence $\gamma_\epsilon$ would be constant, contradicting its optimality. With this fact, we can apply \cite[Theorem~5.2.3]{Clarke} to the autonomous augmented system $\frac{d}{dt}\bigl(t, \gamma_\epsilon(t)\bigr) = \bigl(1, k_\epsilon(t, \gamma_\epsilon(t)) u_\epsilon(t)\bigr)$ and deduce \eqref{eq:cons-Pontryagin} from its conclusions.

As a consequence of Lemma~\ref{Conse Pontryagin}, we obtain the following properties of optimal controls for $\OCP_\epsilon(k_\epsilon)$.

\begin{lemma}
Under the assumption and notations of Lem\-ma~\ref{Conse Pontryagin}, for every $t \in [t_0, t_0 + T_\epsilon]$, we have $\abs{p_{\epsilon}(t)} \neq 0$ and $u_\epsilon(t) = \frac{p_\epsilon(t)}{\abs{p_\epsilon(t)}}$. As a consequence, $u_\epsilon$ is Lipschitz continuous and $\gamma_\epsilon$ is $C^{1, 1}$, and the Lipschitz constant of $u_\epsilon$ depends only on $\epsilon$, $K_{\max}$, and the Lipschitz constant of $k$.
\end{lemma}

\begin{proof}
Let $L$ be the Lipschitz constant of $k$. From the definition of $k_\epsilon$ and standard properties of Clarke's gradient (see, e.g., \cite[Proposition~2.1.2]{Clarke}), we have that $\abs{\zeta} \leq L + \frac{K_{\max}}{\epsilon}$ for every $(t, x) \in \mathbb R_+ \times \mathbb R^d$ and $\zeta \in \pi_2 \partial^C k_\epsilon(t, x)$. Hence, integrating \eqref{eq:cons-Pontryagin:p-epsilon}, we deduce that, for every $t, t_1 \in [t_0, t_0+T_{\epsilon}]$,
$$
\abs{p_{\epsilon}(t)} \le \abs{p_{\epsilon}(t_1)}+\left(L+\frac{K_{max}}{\epsilon}\right) \int_{\min\{t,t_1\}}^{\max\{t,t_1\}} \abs{p_{\epsilon}(s)} ds.
$$
Hence, by Grönwall's inequality, for every $t, t_1 \in [t_0,t_0+T_{\epsilon}]$,
$$
\abs{p_{\epsilon}(t)} \le \abs{p_{\epsilon}(t_1)} e^{\left(L+\frac{K_{max}}{\epsilon}\right)\abs{t-t_1}}.
$$
If there exists $t_1 \in [t_0, t_0 + T_\epsilon]$ such that $p_{\epsilon}(t_1)=0$, then $p_{\epsilon}(t)=0$ for every $t\in [t_0,t_0+T_{\epsilon}]$. Thus, by \eqref{eq:cons-Pontryagin:hamiltonian}, $q_{\epsilon}(t) = -\lambda_{\epsilon}$ for every $t\in [t_0,t_0+T_{\epsilon}]$, and since $q_{\epsilon}(t_0 + T_{\epsilon})=0$ by \eqref{eq:cons-Pontryagin:finalTime}, it follows that $\lambda_{\epsilon} = 0$, which contradicts \eqref{eq:cons-Pontryagin:nontrivial}, establishing thus that $\abs{p_{\epsilon}(t)}\neq 0$ for every $t\in [t_0,t_0+T_{\epsilon}]$.

Thanks to this fact, one deduces immediately from \eqref{eq:cons-Pontryagin:maximization} that $u_\epsilon(t) = \frac{p_\epsilon(t)}{\abs{p_\epsilon(t)}}$. Denoting by $b_\epsilon: [t_0, t_0 + T_\epsilon] \to \mathbb R^d$ a measurable function such that $-\dot p_\epsilon(t) = \abs{p_\epsilon(t)} b_\epsilon(t)$ for a.e.\ $t \in [t_0, t_0 + T_\epsilon]$, we deduce that, for a.e.\ $t \in [t_0, t_0 + T_\epsilon]$,
\[
\dot u_\epsilon(t) = - b_\epsilon(t) + (u_\epsilon(t) \cdot b_\epsilon(t)) u_\epsilon(t).
\]
Since $b_\epsilon(t) \in \pi_2 \partial^C k_\epsilon(t, \gamma_\epsilon(t))$ for a.e.\ $t \in [t_0, t_0 + T_\epsilon]$, we conclude that $\abs{\dot u_\epsilon(t)} \leq L + \frac{K_{\max}}{\epsilon}$, showing that $u$ is Lipschitz continuous, as required.
\end{proof}

Similarly to \cite{CannarsaCastelpietraCardaliaguet}, we now establish the main link between $\OCP(k)$ and $\OCP_{\epsilon}(k_\epsilon)$.

\begin{proposition}\label{PropEquivEpsilon}
Consider $\OCP(k)$ under the assumptions \ref{HypoOmega-GeoDist}, \ref{HypoOCP-k}, and \ref{HypoOmega-C11}, as well as the problem $\OCP_\epsilon(k_\epsilon)$ with $k_\epsilon$ defined by \eqref{eq:k-eps}. There exists $\epsilon_0 > 0$ such that, for every $\epsilon \in (0, \epsilon_0)$ and $(t_0,x_0)\in \mathbb R\times \bar{\Omega}$, the following properties hold.
\begin{enumerate}
\item\label{equiv-OCP-1} If $\gamma_\epsilon$ is an optimal trajectory for $\OCP_\epsilon(k_\epsilon)$ starting from $(t_0, x_0)$, then $\gamma_\epsilon(t) \in \bar\Omega$ for every $t \geq 0$.
\item\label{equiv-OCP-2} If $\gamma \in \Opt(k, t_0, x_0)$, then $\gamma$ is an optimal trajectory for $\OCP_\epsilon(k_\epsilon)$.
\end{enumerate}
As a consequence, if $\gamma \in \Opt(k, t_0, x_0)$ and $u$ is its associated optimal control, then $\gamma$ is $C^{1, 1}$ and $u$ is Lipschitz continuous on $[t_0, t_0 + \tau_\Gamma(t_0, \gamma)]$, and the Lipschitz constant of $u$ depends only on $\epsilon_0$, $K_{\max}$, and the Lipschitz constant of $k$.
\end{proposition}

\begin{proof}
To prove \ref{equiv-OCP-1}, let $T_\epsilon$ be the first exit time of $\gamma_\epsilon$ and assume, to obtain a contradiction, that there exist $a, b \in [t_0, t_0 + T_\epsilon]$ such that $a < b$, $\gamma_\epsilon(t) \notin \bar\Omega$ for $t \in (a, b)$, and $\gamma_\epsilon(t) \in \partial\Omega$ for $t \in \{a, b\}$ (recall that $\Gamma \subset \bar\Omega$ and $x_0 \in \bar\Omega$, so $\gamma_\epsilon$ starts and ends its movement in $\bar\Omega$). The map $t \mapsto d^{\pm}(\gamma_\epsilon(t))$ is differentiable in a neighborhood of $[a, b]$, strictly positive for $t \in (a, b)$, and equal to $0$ for $t \in \{a, b\}$, and thus its derivative is nonnegative at $a$ and nonpositive at $b$, i.e.,
\[
\dot\gamma_\epsilon(a) \cdot \nabla d^{\pm} (\gamma_\epsilon(a)) \geq 0, \qquad \dot\gamma_\epsilon(b) \cdot \nabla d^{\pm} (\gamma_\epsilon(b)) \leq 0.
\]
Since $\dot\gamma_\epsilon(t) = k_\epsilon(t, \gamma_\epsilon(t)) \frac{p_\epsilon(t)}{\abs{p_\epsilon(t)}}$ and $\frac{k_\epsilon(t, \gamma_\epsilon(t))}{\abs{p_\epsilon(t)}} > 0$ for every $t \in [t_0, t_0 + T_\epsilon]$, we have
\begin{equation}
\label{AlphaInAandB}
p_\epsilon(a) \cdot \nabla d^{\pm}(\gamma_\epsilon(a)) \geq 0, \qquad p_\epsilon(b) \cdot \nabla d^{\pm}(\gamma_\epsilon(b)) \leq 0.
\end{equation}

Consider the map $\alpha: t \mapsto p_\epsilon(t) \cdot \nabla d^{\pm}(\gamma_\epsilon(t))$. Since $d^{\pm}$ is $C^{1, 1}$ in a neighborhood of $\partial\Omega$ and $d(\gamma_\epsilon(t), \Omega) \leq \epsilon$ for every $t \in [t_0, t_0 + T_\epsilon]$ by Lemma~\ref{Conse Pontryagin}, if $\varepsilon_0 > 0$ is small enough, we deduce that $\gamma_\epsilon(t)$ belongs to the neighborhood at which $d^{\pm}$ is $C^{1, 1}$ for every $t \in (a, b)$. Thus $\alpha$ is absolutely continuous on $[a, b]$ and, recalling that $k$ is $C^1$ on $\mathbb R_+ \times (\mathbb R^d \setminus \bar\Omega)$ and using \eqref{eq:cons-Pontryagin:p-epsilon}, we have, for $t \in (a, b)$,
\begin{multline*}
\dot\alpha(t) = \dot p_\epsilon(t) \cdot \nabla d^{\pm}(\gamma_\epsilon(t)) + p_\epsilon(t) \cdot \frac{d \left[\nabla d^{\pm} \circ \gamma_\epsilon\right]}{d t}(t) \displaybreak[0] \\
 = - \abs{p_\epsilon(t)} \left(1 - \frac{1}{\epsilon} d(\gamma_\epsilon(t), \Omega)\right)_+ \nabla_x k(t, \gamma_\epsilon(t)) \cdot \nabla d^{\pm}(\gamma_\epsilon(t)) \\
 {} + \abs{p_\epsilon(t)} \frac{1}{\epsilon} k(t, \gamma_\epsilon(t)) \abs{\nabla d^{\pm}(\gamma_\epsilon(t))}^2 + p_\epsilon(t) \cdot \frac{d \left[\nabla d^{\pm} \circ \gamma_\epsilon\right]}{d t}(t) \displaybreak[0] \\
\geq \abs{p_\epsilon(t)} \left[-L + \frac{K_{\min}}{\epsilon} - L K_{\max}\right],
\end{multline*}
where $L$ is an upper bound on the Lipschitz constants of $d^{\pm}$ and $k$. Up to decreasing $\epsilon_0$, we have $-L + \frac{K_{\min}}{\epsilon} - L K_{\max} > 0$ for every $\epsilon \in (0, \epsilon_0)$,
and hence $\dot\alpha(t) > 0$ for $t \in (a, b)$, which contradicts \eqref{AlphaInAandB}. This contradiction establishes \ref{equiv-OCP-1}.

To establish \ref{equiv-OCP-2}, let $\gamma_\epsilon$ be an optimal trajectory for $\OCP_\epsilon(k_\epsilon)$ starting at $(t_0, x_0)$ and denote by $T_\epsilon$ its first exit time after $t_0$. Since $\gamma \in \Opt(k, t_0, x_0)$, $\gamma$ is admissible for $\OCP_\epsilon(k_\epsilon)$, and thus $\tau_\Gamma(t_0, \gamma) \geq T_\epsilon$. On the other hand, by \ref{equiv-OCP-1}, we have $\gamma_\epsilon \in \Adm(k)$, and thus $T_\epsilon \leq \tau_\Gamma(t_0, \gamma)$. Thus $\tau_\Gamma(t_0, \gamma) = T_\epsilon$, concluding the proof.
\end{proof}

\subsection{Boundary condition of the Hamilton--Jacobi equation}

Having established in particular in Proposition~\ref{PropEquivEpsilon} that optimal controls for $\OCP(k)$ are Lipschitz continuous, we are now able to deduce a boundary condition for the Hamilton--Jacobi equation \eqref{H-J equation}.

\begin{proposition}\label{PropBoundaryHJ}
Consider $\OCP(k)$ under the assumptions \ref{HypoOmega-GeoDist}, \ref{HypoOCP-k}, and \ref{HypoOmega-C11}, and let $\varphi$ be its value function and $\mathbf n$ be the exterior normal of $\Omega$. Then $\varphi$ satisfies $\nabla\varphi(t, x) \cdot \mathbf n(x) \geq 0$ for $(t, x) \in \mathbb R_+ \times (\partial\Omega \setminus \Gamma)$ in the viscosity supersolution sense.
\end{proposition}

\begin{proof}
Let $(t_0, x_0) \in \mathbb R_+ \times (\partial\Omega \setminus\Gamma)$ and $\xi$ be a smooth function defined on a neighborhood $V$ of $(t_0, x_0)$ in $\mathbb R_+ \times \bar\Omega$ such that $\xi(t_0, x_0) = \varphi(t_0, x_0)$ and $\xi(t, x) \leq \varphi(t, x)$ for $(t, x) \in V$. Assume, to obtain a contradiction, that $\nabla\xi(t_0, x_0) \cdot \mathbf n(x_0) < 0$. Let $\gamma \in \Opt(k, t_0, x_0)$, denote by $u$ its associated optimal control, and define $\tilde\gamma: [t_0 - \epsilon, +\infty) \to \bar\Omega$ for $\epsilon > 0$ small enough by $\tilde\gamma(t) = \gamma(t)$ for $t \geq t_0$ and as the solution of $\dot{\tilde\gamma}(t) = - k(t, \tilde\gamma(t)) \frac{\nabla\xi(t_0, x_0)}{\abs{\nabla\xi(t_0, x_0)}}$ for $t \in [t_0 - \epsilon, t_0]$ with final condition $\tilde\gamma(t_0) = x_0$ (we extend $k$ to negative times in a Lipschitz manner if needed). Applying Proposition~\ref{PropOCP}\ref{PropOCP-DPP} to $\tilde\gamma$, we get that $\varphi(t_0, x_0) \geq \varphi(t_0 - h, \tilde\gamma(t_0 - h)) - h$ for every $h \in [0, \epsilon]$, and thus $\xi(t_0, x_0) \geq \xi(t_0 - h, \tilde\gamma(t_0 - h)) - h$. Since $\xi(t_0 - h, \tilde\gamma(t_0 - h)) = \xi(t_0, x_0) - h \partial_t \xi(t_0, x_0) - h \dot{\tilde\gamma}(t_0^-) \cdot \nabla\xi(t_0, x_0) + o(h)$, we deduce that
\begin{equation}
\label{eq:viscosity-boundary-1}
-\partial_t \xi(t_0, x_0) + k(t_0, x_0) \abs{\nabla\xi(t_0, x_0)} - 1 \leq 0.
\end{equation}

Since $\gamma \in \Opt(k, t_0, x_0)$, we have, by Proposition~\ref{PropOCP}\ref{PropOCP-DPP}, that $\varphi(t_0, x_0) = \varphi(t_0 + h, \gamma(t_0 + h)) + h$ for $h \geq 0$ small enough, and thus $\xi(t_0, x_0) \geq \xi(t_0 + h, \tilde\gamma(t_0 + h)) + h$. Performing the first order expansion of $\xi(t_0 + h, \gamma(t_0 + h))$ on $h$ as before and using the fact that $\gamma$ satisfies \eqref{Control sys}, we deduce that $\partial_t \xi(t_0, x_0) + k(t_0, x_0) \nabla\xi(t_0, x_0) \cdot u(t_0) + 1 \leq 0$. Adding with \eqref{eq:viscosity-boundary-1}, we deduce that $\nabla\xi(t_0, x_0) \cdot u(t_0) + \abs{\nabla\xi(t_0, x_0)} \leq 0$ and, since $u(t_0) \in B_1$, this implies that $u(t_0) = -\frac{\nabla\xi(t_0, x_0)}{\abs{\nabla\xi(t_0, x_0)}}$. Since $\nabla\xi(t_0, x_0) \cdot \mathbf n(x_0) < 0$, this would imply that $\gamma$ leaves $\bar\Omega$ at some $t_0 + h$ for $h > 0$ small enough, contradicting the fact that $\gamma \in \Opt(k, t_0, x_0)$. This contradiction establishes that $\nabla\xi(t_0, x_0) \cdot \mathbf n(x_0) \geq 0$, as required.
\end{proof}

\subsection{Characterization of optimal controls}

Using Propositions~\ref{PropEquivEpsilon} and \ref{PropBoundaryHJ}, we are now in position to characterize optimal controls of $\OCP(k)$. We start by introducing the two main objects that we will use in our characterization.

\begin{definition}
Consider $\OCP(k)$ under assumptions \ref{HypoOmega-GeoDist}, \ref{HypoOCP-k}, and \ref{HypoOmega-C11}. Let $\varphi$ be its value function and take $(t_0,x_0) \in \mathbb R_+ \times \bar{\Omega}$.
\begin{enumerate}
\item We define the set $\mathcal{U}(t_0,x_0)$ of \emph{optimal directions} at $(t_0,x_0)$ as the set of $u_0 \in \mathbb S^{d-1}$ for which there exists $\gamma \in \Opt(k, t_0, x_0)$ such that the optimal control $u$ associated with $\gamma$ satisfies $u(t_0) = u_0$.
\item We define the set $\mathcal W(t_0, x_0)$ of \emph{directions of maximal descent of $\varphi$} at $(t_0, x_0)$ as the set of $u_0 \in \mathbb S^{d-1}$ such that
\begin{equation}
\label{eq:defi-W}
\lim_{h\to 0^+} \frac{\varphi(t_0 + h, x_0 + h k(t_0, x_0) u_0) - \varphi(t_0, x_0)}{h} = -1.
\end{equation}
\end{enumerate}
\end{definition}

Note that $\mathcal U(t_0, x_0) \neq \varnothing$ for $x_0 \in \bar\Omega \setminus \Gamma$ and, by Proposition~\ref{PropOCP}\ref{PropOCP-DPP}, the quantity on the left-hand side of \eqref{eq:defi-W} whose limit is being computed is greater than or equal to $-1 + o(1)$ as $h \to 0^+$. The main result of this section is the following.

\begin{theorem}\label{thm Ut_0,x_0}
Consider $\OCP(k)$ and its value function $\varphi$ under \ref{HypoOmega-GeoDist}, \ref{HypoOCP-k}, and \ref{HypoOmega-C11} and let $(t_0, x_0) \in \mathbb R_+ \times \bar\Omega$.
\begin{enumerate}
\item\label{UW-3} If $\varphi$ is differentiable at $(t_0, x_0)$, then $\mathcal W(t_0, x_0) = \left\{-\frac{\nabla\varphi(t_0, x_0)}{\abs{\nabla\varphi(t_0, x_0)}}\right\}$.
\item\label{UW-2} For every $\gamma \in \Opt(k, t_0, x_0)$ and $t \in (t_0, t_0 + \varphi(t_0, x_0))$, $\mathcal U(t, \gamma(t))$ contains exactly one element.
\item\label{UW-1} We have $\mathcal{U}(t_0,x_{0}) = \mathcal{W}(t_0,x_{0})$.
\end{enumerate}
\end{theorem}

\begin{proof}
Assertion \ref{UW-3} follows easily from \eqref{eq:defi-W} by using Proposition~\ref{PropOCP}\ref{PropOCP-HJ} and \ref{PropOCP-OptimalControl} (see also \cite[Proposition~4.13]{SadeghiMulti} for a proof in the case with no state constraints). Assertion \ref{UW-2} follows from the fact that optimal controls are Lipschitz continuous (Proposition~\ref{PropEquivEpsilon}) and its proof is very similar to that of \cite[Proposition~4.7]{Mazanti2019Minimal}. As for assertion \ref{UW-1}, its proof is very similar to that of \cite[Theorem~4.14]{SadeghiMulti} and we sketch it here for completeness. First, note that it suffices to consider the case $x_0 \in \bar\Omega \setminus \Gamma$ since both sets are empty if $x_0 \in \Gamma$. The inclusion $\mathcal U(t_0, x_0) \subset \mathcal W(t_0, x_0)$ can be obtained by applying Proposition~\ref{PropOCP}\ref{PropOCP-DPP} and taking the limit as $h \to 0^+$ in \eqref{eq:defi-W}. For the converse inclusion, let $u_0 \in \mathcal W(t_0, x_0)$ and note that, if $x_0 \in \partial\Omega \setminus\Gamma$, then necessarily $u_0$ points towards the inside of $\Omega$. Let $\gamma_0$ be the solution of \eqref{Control sys} starting from $(t_0, x_0)$ and with constant control $u_0$, defined in $[t_0, t_0 + h]$ for some $h > 0$ small enough, define $t_1 = t_0 + h$ and $x_1 = \gamma_0(t_0 + h)$, and let $\gamma_1 \in \Opt(k, t_1, x_1)$. The conclusion follows by letting $h \to 0^+$ if one assumes that the optimal control $u_1$ associated with $\gamma_1$ satisfies $u_1(t_1) \to u_0$ as $h \to 0^+$, using the fact that limits of optimal trajectories are also optimal.

We prove by contradiction that we necessarily have $u_1(t_1) \to u_0$ as $h \to 0^+$. Indeed, assume that this is not the case, let $\bar\gamma_1$ be the solution of \eqref{Control sys} starting from $(t_1, x_1)$ and with constant control $u_1(t_1)$, and define $t_2 = t_1 + h$, $x_2 = \gamma_1(t_2)$, and $\bar x_2 = \bar\gamma_1(t_2)$. Define also $\gamma_2$ as the solution of \eqref{Control sys} starting from $(t_0, x_0)$ and with constant control $\frac{\bar x_2 - x_0}{\abs{\bar x_2 - x_0}}$, $\tau$ be the time at which $\gamma_2$ arrives at $\bar x_2$, and $\gamma_3$ be the solution of \eqref{Control sys} starting from $(\tau, \bar x_2)$ and with constant control $\frac{x_2 - \bar x_2}{\abs{x_2 - \bar x_2}}$ (see Figure~\ref{Fig1} for an illustration of these constructions). Note that, since $u_0$ points towards the inside of $\Omega$, all points and trajectories in this construction remain in $\bar\Omega$ for $h$ small enough.

\tikzset{
 mid arrow/.style={postaction={decorate,decoration={
        markings,
        mark=at position .5 with {\arrow{Stealth}}
      }}}
}

\begin{figure}[ht]
\centering
\resizebox{0.5\columnwidth}{!}{
\begin{tikzpicture}
\node (x0) at (0, 0) {};
\node (x1) at (0, -1) {};
\node (x2bar) at (3, -1) {};
\node (x2) at (3.25, -2) {};

\draw[red, mid arrow] (x0.center) -- node[midway, left] {$\gamma_0$} (x1.center);
\draw[blue, mid arrow] (x0.center) -- node[midway, above right] {$\gamma_2$} (x2bar.center);
\draw[violet, mid arrow] (x1.center) to[out = 0, in = 180] node[midway, below left] {$\gamma_1$} (x2.center);
\draw[violet, mid arrow] (x2.center) -- ++(1, 0);
\draw[green!50!black, mid arrow] (x1.center) -- node[midway, below right] {$\bar \gamma_1$} (x2bar.center);
\draw[orange!75!black, mid arrow] (x2bar.center) -- node[midway, right] {$\gamma_3$} (x2.center);

\fill (x0) circle[radius=0.05] node[left] {$x_0$};
\fill (x1) circle[radius=0.05] node[left] {$x_1$};
\fill (x2bar) circle[radius=0.05] node[right] {$\bar x_2$};
\fill (x2) circle[radius=0.05] node[below] {$x_2$};
\end{tikzpicture}}
\caption{Illustration of the constructions used in the proof of Theorem~\ref{thm Ut_0,x_0}\ref{UW-1} (adapted from \cite{SadeghiMulti}).}
\label{Fig1}
\end{figure}
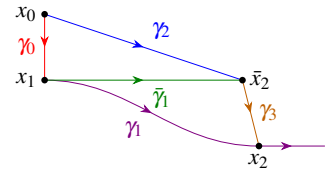

Since the angle between $\gamma_0$ and $\bar\gamma_1$ at $x_1$ is different from $\pi$ as $h \to 0^+$, one can prove that there exists $\rho < 1$ such that the time $\tau - t_0$ that $\gamma_2$ takes to go from $x_0$ to $\bar x_2$ is at most $2 \rho h + O(h^2)$. On the other hand, since $\gamma_1$ and $\bar\gamma_1$ are tangent at $x_1$, the time that $\gamma_3$ takes to go from $\bar x_2$ to $x_2$ is at most $O(h^2)$. We have thus constructed two trajectories to go from $x_0$ to $x_2$: one obtained as the concatenation of $\gamma_0$ and $\gamma_1$, which takes a time $2h$, and another obtained as the concatenation of $\gamma_2$ and $\gamma_3$, which takes a time $2 \rho h + O(h^2) < 2h$. By applying Proposition~\ref{PropOCP}\ref{PropOCP-DPP} to both trajectories, letting $h \to 0^+$, and using \cite[Proposition~4.4]{Mazanti2019Minimal}, one gets the conclusion that $\rho \geq 1$, yielding the desired contradiction.
\end{proof}

Motivated by Theorem~\ref{thm Ut_0,x_0}\ref{UW-3}, we provide the following definition.

\begin{definition}
Consider $\OCP(k)$ under assumptions \ref{HypoOmega-GeoDist}, \ref{HypoOCP-k}, and \ref{HypoOmega-C11}, let $\varphi$ be its value function, and take $(t_0, x_0) \in \mathbb R_+ \times \bar\Omega$. If $\mathcal{W}(t_0, x_0)$ contains exactly one element, we denote this element by $-\widehat{\nabla\varphi}(t_0,x_0)$, and call it the \emph{normalized gradient} of $\varphi$ at $(t_0, x_0)$.
\end{definition}

As a consequence of Theorem~\ref{thm Ut_0,x_0}, we obtain the following characterization of optimal trajectories.

\begin{corollary}\label{coro normalized}
Consider $\OCP(k)$ under assumptions \ref{HypoOmega-GeoDist}, \ref{HypoOCP-k}, and \ref{HypoOmega-C11}, let $\varphi$ be its value function, and take $(t_0, x_0) \in \mathbb R_+ \times \bar\Omega$ and $\gamma \in \Opt(k, t_0, x_0)$. Then, for every $t \in (t_0, t_{0} + \varphi(t_{0}, x_0))$, $\varphi$ admits a normalized gradient at $(t, \gamma(t))$ and
$
\Dot{\gamma}(t)= -k(t, \gamma(t))  \widehat{\nabla \varphi}(t, \gamma(t)).
$
\end{corollary}

We also have the following result on the normalized gradient, whose proof is very similar to that of \cite[Proposition~4.17]{SadeghiMulti} (see also \cite[Proposition~4.9]{Mazanti2019Minimal}).

\begin{proposition}\label{PropNormalizedGradientContinuous}
Consider $\OCP(k)$ and its value function $\varphi$ under assumptions \ref{HypoOmega-GeoDist}, \ref{HypoOCP-k}, and \ref{HypoOmega-C11}. Then $\widehat{\nabla\varphi}$ is continuous on its domain of definition.
\end{proposition}

\section{The MFG System}\label{sec MFG system}

Following the results on $\OCP(k)$ and using the relation between $\MFG(K, m_0)$ and $\OCP(k)$, we are now ready to obtain, as a consequence of Proposition~\ref{PropOCP}\ref{PropOCP-HJ}, Proposition~\ref{PropBoundaryHJ}, Corollary~\ref{coro normalized}, and Proposition~\ref{PropNormalizedGradientContinuous}, that equilibria of $\MFG(K, m_0)$ satisfy a system of PDEs.

\begin{theorem}\label{Thm MFG system}
Consider $\MFG(K, m_0)$ under the assumptions \ref{HypoOmega-GeoDist}, \ref{HypoMFG-K}, and \ref{HypoOmega-C11}. Let $Q \in \mathcal P(C(\bar\Omega))$ be an equilibrium of $\MFG(K, m_0)$, $m_t = {e_t}_{\#}Q$ for $t \geq 0$, $k$ be defined from $K$ by $k(t, x) = K(m_t, x)$, and $\varphi$ be the value function of $\OCP(k)$. Then, $(m_t, \varphi)$ solves the MFG system
\begin{equation}\label{MFGs system}
    \left \{
    \begin{aligned}
    & \partial_t m_t-\diverg\bigl(m_t K(m_t,x) \widehat{\nabla \varphi}\Bigr)=0 & & \text{in } \mathbb R_+^* \times (\bar{\Omega} \setminus \Gamma), \\
    & -\partial_t \varphi + \abs{\nabla \varphi} K(m_t,x) - 1 = 0 & & \text{in } \mathbb R_+\times (\bar{\Omega} \setminus \Gamma), \\
    & \varphi = 0 & & \text{on } \mathbb R_+\times \Gamma, \\
    & \nabla\varphi \cdot \mathbf n \ge 0 & & \text{on } \mathbb R_+ \times (\partial\Omega \setminus \Gamma),
    \end{aligned} \right.
\end{equation}
where the first equation is satisfied in the sense of distributions and the second and fourth equations are satisfied in the viscosity senses of Propositions~\ref{PropOCP}\ref{PropOCP-HJ} and \ref{PropBoundaryHJ}, respectively. In addition, $m_t\lvert_{t = 0} = m_0$ and $m_t\lvert_{\partial\Omega_t^-} = 0$, where $\partial\Omega_t^-$ is the part of $\partial\Omega$ at which $\widehat{\nabla \varphi} \cdot \mathbf n > 0$.
\end{theorem}

\bibliographystyle{IEEEtran}
\bibliography{main}

\end{document}